\documentclass{ws-ijnt1}

\newcommand{\Q}{\mathbb{Q}}

\newcommand{\cO}{\mathcal{O}}
\newcommand{\fkp}{\mathfrak{p}}
\newcommand{\ve}{\varepsilon}

\renewcommand{\leq}{\leqslant}
\renewcommand{\geq}{\geqslant}
\newcommand{\rr}{\rightarrow}
\newcommand{\of}[1]{\left(#1\right)}
\newcommand{\set}[1]{\left\{#1\right\}}
\newcommand{\abs}[1]{\left\vert#1\right\vert}
\newcommand{\smat}[1]{\begin{smallmatrix}#1\end{smallmatrix}}
\newcommand{\on}{\operatorname}
\renewcommand{\Re}{\on{Re}}

\newcommand{\Gal}{\on{Gal}}

\begin{document}

\markboth{Biao Wang}{An analogue of a formula for Chebotarev Densities}

\title{An analogue of a formula for Chebotarev Densities}

\author{Biao Wang}

\address{Department of Mathematics\\
	University at Buffalo, The State University of New York\\
 	Buffalo, NY 14260, USA\\
\email{bwang32@buffalo.edu} }

\maketitle

\begin{abstract}
	In this short note, we show an analogue of Dawsey's formula on Chebotarev densities for finite Galois extensions of $\Q$ with respect to the Riemann zeta function $\zeta(ms)$ for any integer $m\geq2$. Her formula may be viewed as the limit version of ours as $m\rr\infty$. 
\end{abstract}



\section{Introduction and statement of results}

Let $\zeta(s)=\sum_{n=1}^\infty\frac1{n^s}$ for $\Re s>1$ be the Riemann zeta function, and let $\mu(n)$ be the M\"{o}bius function defined by $\mu(n)=(-1)^k$ if $n$ is the product of $k$ distinct primes and zero otherwise. 
It is well-known (e.g.,  \cite[(4.5)]{d82}) that the prime number theorem is equivalent to the assertion that 
\begin{equation}\label{pntmu}
\sum_{n=1}^\infty\frac{\mu(n)}{n}=0
\end{equation}
or equivalently,
\begin{equation}
-\sum_{n=2}^\infty\frac{\mu(n)}{n}=1.
\end{equation}

Let $p(n)$ be the smallest prime divisor of $n$ and let $\varphi$ be the Euler totient function.  Let $k\geq 1$, $\ell$ be integers and $(\ell,k)=1$. In 1977, Alladi \cite{a77} proved that 
\begin{equation}\label{alladi}
-\sum_{\smat{n\geq 2\\ p(n)\equiv \ell (\on{mod}k)}}\frac{\mu(n)}{n}=\frac1{\varphi(k)}.
\end{equation}

In 2017, Dawsey \cite{d17} generalized formula (\ref{alladi}) to the setting of Chebotarev densities for finite Galois extensions
of $\Q$. That is, for any conjugacy class $C$ in the Galois group $G = \Gal(K/\Q)$ of a finite  Galois extension $K$ of $\Q$, we have
\begin{equation}\label{dawsey}
-\sum_{\smat{n\geq 2\\ \left[\frac{K/\Q}{p(n)}\right]=C}}\frac{\mu(n)}{n}=\frac{|C|}{|G|},
\end{equation}
where  
$$\left[\frac{K/\Q}{p}\right]:=\set{\left[\frac{K/\Q}{\fkp}\right]: \fkp \subseteq\cO_K  \text{ and } \fkp|p}$$ 
for an unramified prime $p$, and $\left[\frac{K/\Q}{\fkp}\right]$ is the Artin symbol for the Frobenius map. Here $\cO_K$ denotes the ring of integers in $K$, and $\fkp$ denotes a prime ideal in $\cO_K$.

Alladi's result (\ref{alladi}) is the special case of (\ref{dawsey}) when $K=\Q(\zeta_k)$ and $C$ is the conjugacy class of $\ell$, where $\zeta_k$ is a primitive $k$-th root of unity.

In this note, we give an analogue of Alladi's and Dawsey's results relating to $\zeta(ms)$ for any integer $m\geq2$. Let $\lambda_m(n)$ be the function defined as the coefficient of term $\frac{1}{n^s}$ in the Dirichlet series expansion of $\frac{\zeta(ms)}{\zeta(s)}$ for $\Re s>1$. That is,
\begin{equation}\label{lddfn}
\sum_{n=1}^\infty \frac{\lambda_m(n)}{n^s}=\frac{\zeta(ms)}{\zeta(s)}
\end{equation}
for $\Re s>1$. 
 When $m=2$, $\lambda_2(n)=(-1)^{\Omega(n)}$ is the Liouville function (e.g., \cite[Theorem 300]{hw08}), where $\Omega(n)=\sum_{p^\alpha||n}\alpha$. Hence $\lambda_m(n)$ is a generalization of the Liouville function. In section \ref{pntld}, we will see that 
$
 	\lambda_m(n)=\sum_{d^m|n}\mu\big(\frac{n}{d^m}\big),
$
 and  the prime number theorem is equivalent to  the assertion that 
\begin{equation}\label{pntlf}
\sum_{n=1}^\infty\frac{\lambda_m(n)}{n}=0.
\end{equation}

Analogous to Alladi's formula (\ref{alladi}), for $(\ell,k)=1$ we have that
\begin{equation}\label{classical}
-\sum_{\smat{n\geq 2\\ p(n)\equiv \ell (\on{mod}k)}}\frac{\lambda_m(n)}{n}=\frac1{\varphi(k)}.
\end{equation}
As \cite{d17}, Eq. (\ref{classical}) can be thought of as a special case in the following main theorem.

\begin{theorem}\label{mainthm}
	Let $K$ be a finite Galois extension of $\Q$ with Galois group $G = \Gal(K/\Q)$. Then for any conjugacy class $C\subseteq G$, we have  
	\begin{equation}\label{maineq}
	-\sum_{\smat{n\geq 2\\ \left[\frac{K/\Q}{p(n)}\right]=C}}\frac{\lambda_m(n)}{n}=\frac{|C|}{|G|}.
	\end{equation}
\end{theorem}

\begin{remark}
	Since $\lim\limits_{m\rr\infty}\zeta(ms)=1$ for $s>1$, we have $\lim\limits_{m\rr\infty}\lambda_m(n)=\mu(n).$ Hence Alladi's and Dawsey's results may be viewed as the limit version of  (\ref{classical}) and (\ref{maineq}), respectively.
\end{remark}
\begin{remark}
	In 2019, Sweeting and Woo \cite{sw19}  generalized  (\ref{dawsey}) to finite Galois extensions
	of number fields. One may also generalize  (\ref{maineq}) to number fields.
\end{remark}

For the proof of Theorem \ref{mainthm}, we shall use a prime divisor function $P_m(n)$ which will be defined in section \ref{dualsec} to estimate the difference between the partial sums of  (\ref{dawsey}) and (\ref{maineq}). As a result, $P_m(n)$ is very close to the largest prime divisor function $P(n)$ and satisfies Alladi's duality property. Then we apply Dawsey's result in \cite{d17}.

\section{Some properties of $\lambda_m(n)$}\label{pntld}

In this section, we mainly introduce the relation between $\lambda_m$ and $\mu$ and prove the prime number theorem with respect to $\lambda_m$. 

\begin{lemma} \label{ldlemma}
	Let $m\geq2$ be a fixed integer. For the $\lambda_m$ defined by (\ref{lddfn}), we have
	\begin{enumerate}
		\item $\lambda_m$ is a multiplicative function.
		
		\item $\lambda_m(n)=\sum_{d^m|n}\mu\big(\frac{n}{d^m}\big).$
		
		\item For any integer $n\geq1$, we can write it as $n=k^m\cdot l$ for $k,l\geq1$ and $l$ is $m$-th power-free (i.e., it has no $m$-th power divisor except 1). Then $\lambda_m(n)=\mu(l)$.
		
		\item $\mu(n)=\mu^2(n)\lambda_m(n)$ for all integers $n\geq1$.
	\end{enumerate}	
\end{lemma}
\begin{proof} Set
	\begin{equation}\label{andfn}
		a(n):=\begin{cases}
			1,& \text{if } n=d^m \text{ for some integer } d\geq1;\\
			0, & \text{otherwise}.
		\end{cases}	
	\end{equation}
	Then $a(n)$ is multiplicative and $\sum_{n=1}^\infty\frac{a(n)}{n^s}=\zeta(ms)$ for $\Re s>1$. 
		\begin{enumerate}
\item It is well known (e.g. \cite[Corollary 11.3]{fkl18}) that $\frac1{\zeta(s)}=\sum_{n=1}^\infty\frac{\mu(n)}{n^s}$ for $\Re s>1$. By (\ref{lddfn}), the definition  of $\lambda_m(n)$,  for $\Re s>1$ we have
	\begin{equation}\label{ldeq-1}
		\sum_{n=1}^\infty \frac{\lambda_m(n)}{n^s}=\sum_{n=1}^\infty\frac{a(n)}{n^s}\sum_{n=1}^\infty\frac{\mu(n)}{n^s}.
	\end{equation}
 It follows that $\lambda_m=a*\mu$ is the Dirichlet convolution of $a$ and $\mu$, which are both multiplicative functions. Hence $\lambda_m$ is multiplicative. 
	
\item	Since $\lambda_m=a*\mu$, we have
\begin{equation}\label{ldeq-2}
	\lambda_m(n)=\sum_{d|n}a(d)\mu\big(\frac{n}{d}\big).
\end{equation}
Plugging (\ref{andfn}) into (\ref{ldeq-2}), we get part (2).

\item Since $\lambda_m$ is multiplicative, it suffices to consider the prime powers. Suppose $n=p^\alpha$, $\alpha\geq1$. Write $\alpha$ as $\alpha=m\beta+r$ with integers $\beta\geq0$ and $0\leq r<m$. Then $p^\alpha=(p^\beta)^m\cdot p^r$ and we can use part (2) to compute $\lambda_m(p^\alpha)$ as follows:
$$
	\lambda_m(p^\alpha)=\sum_{d^m|p^\alpha}\mu\big(\frac{p^\alpha}{d^m}\big)=\sum_{j=0}^\beta \mu\big(\frac{p^\alpha}{p^{mj}}\big)=\sum_{j=0}^\beta \mu(p^{m(\beta-j)+r})=\mu(p^r).
$$

\item 
	
	By part (3), $\lambda_m(n)=\mu(n)$ if $n$ is square-free. Then part (4) follows immediately by the fact that $\mu$ is supported on square-free numbers.
		\end{enumerate}	
\end{proof}
\begin{remark}
	 Due to Lemma \ref{ldlemma}(2), analogous to the M\"{o}bius function $\mu(n)$, the Riemann hypothesis is equivalent to the estimate that for all $\epsilon>0$ we have
	\begin{equation}\label{rhlf}
		\sum_{n\leq x}\lambda_m(x)=O(x^{\frac12+\epsilon})
	\end{equation}
	where the implied constant depends on $\epsilon$, see  \cite[Theorem 4.16, 4.18]{b17}. 
\end{remark}
\begin{remark}
	Sarnak’s conjecture with respect to $\mu$ is equivalent to Sarnak’s conjecture with respect to $\lambda_m$ due to Lemma \ref{ldlemma}(2) and (4), see \cite[Corollary 11.25]{fkl18}.
\end{remark}

\begin{lemma}\label{pntlem} 
	The prime number theorem is equivalent to  the assertion that 
	\begin{equation}\label{pntlemeq0}
		\sum_{n=1}^\infty\frac{\lambda_m(n)}{n}=0.
	\end{equation}
\end{lemma}

\begin{proof} Since the prime number theorem is equivalent to (\ref{pntmu}), it suffices to prove that (\ref{pntlemeq0}) is equivalent to (\ref{pntmu}).
	
	First, assume that (\ref{pntmu}) holds. Let 
	$$A(x):=\sum_{n\leq x}\frac{\mu(n)}{n},$$
then $A(x)=o(1)$. By Lemma \ref{ldlemma}(2), we can divide the partial sum of (\ref{pntlemeq0}) into two parts:
	\begin{align}\label{pntlemeq}
		\sum_{n\leq x}\frac{\lambda_m(n)}{n}&=\sum_{n\leq x}\frac1n\sum_{d^me=n}\mu(e)=\sum_{d^m\leq x}\frac1{d^m}A\of{\frac{x}{d^m}}\nonumber\\	
		&=\sum_{d^m\leq x^{\frac1{2}}} \frac1{d^m}A\of{\frac{x}{d^m}}+ \sum_{ x^{\frac12}< d^m\leq x} \frac1{d^m}A\of{\frac{x}{d^m}}. 
	\end{align}
	
	For the first sum, given any $\ve>0$, there exists some $K>0$ such that $|A(x)|<\frac{\ve}{\zeta(m)}$ for all $x>K$. Then for any $x>K^2$, we have $\frac{x}{d^m}\geq x^{\frac12}>K$ for $d^m\leq x^{\frac12}$. So
	\begin{equation}\label{pntlemeq1-0}
		\abs{\sum_{d^m\leq x^{\frac1{2}}} \frac1{d^m}A\of{\frac{x}{d^m}}}\leq \sum_{d^m\leq x^{\frac1{2}}} \frac1{d^m} \cdot\frac{\ve}{\zeta(m)}< \sum_{d=1}^\infty \frac1{d^m} \cdot\frac{\ve}{\zeta(m)}=\ve.
	\end{equation}
This implies that 
	\begin{equation}\label{pntlemeq1}
		\sum_{d^m\leq x^{\frac1{2}}} \frac1{d^m}A\of{\frac{x}{d^m}}=o(1).
	\end{equation}
	
	For the second sum, notice that $A(x)=O(1)$ due to $A(x)=o(1)$. We have that
	\begin{equation}\label{pntlemeq2}
		\sum_{ x^{\frac12}< d^m\leq x} \frac1{d^m}A\of{\frac{x}{d^m}} =O\Big(\sum_{ x^{\frac1{2m}}< d\leq x^{\frac1m}} \frac1{d^m}\Big)=O\big(x^{-\frac{m-1}{2m}}\big)
	\end{equation}
Thus, (\ref{pntlemeq0}) follows by combining (\ref{pntlemeq}), (\ref{pntlemeq1}) and (\ref{pntlemeq2}) together.

   Now, assume that (\ref{pntlemeq0}) holds. First, by the definition of $\lambda_m$, we have $\frac{1}{\zeta(s)}=\frac{1}{\zeta(ms)}\sum_{n=1}^\infty\frac{\lambda_m(n)}{n^s}$ for $\Re s>1$.  Computing the Dirichlet series expansions of this identity and then comparing the coefficients of $n^{-s}$ on both sides, we obtain that
   \begin{equation}\label{muld}
   	\mu(n)=\sum_{d^m|n}\mu(d)\lambda_m\of{\frac{n}{d^m}}.
   \end{equation}
   Then similar to (\ref{pntlemeq}) above, we divide the partial sum of (\ref{pntmu}) into two parts:
   \begin{equation}\label{ldmu-2}
   	\sum_{n\leq x}\frac{\mu(n)}{n}=\sum_{d^m\leq x^{\frac1{2}}} \frac{\mu(d)}{d^m}L\of{\frac{x}{d^m}}+ \sum_{ x^{\frac12}< d^m\leq x} \frac{\mu(d)}{d^m}L\of{\frac{x}{d^m}},
   \end{equation}
   where $L(x)=\sum_{n\leq x}\frac{\lambda_m(n)}{n}$. The two sums on the right side of (\ref{ldmu-2}) are of $o(1)$ by the similar argument of (\ref{pntlemeq1}) and (\ref{pntlemeq2}) due to $|\mu(n)|\leq 1$ for all $n$, and (\ref{pntmu}) follows. This completes the proof.
\end{proof}

\section{Duality of prime factors}\label{dualsec}
\begin{lemma}[Duality Lemma]\label{duality}
	For any arithmetic function $f(n)$ with $f(1)=0$, we have
	\begin{equation}\label{dualeq}
	\sum_{d|n}\lambda_m(d)f(p(d))=-f(P_m(n))
	\end{equation}
	where $p(1)=1$ and $P_m(n)$ is the largest prime factor of $n$ of order $\not\equiv0(\on{mod} m)$ and is $1$ if $n$ is a perfect $m$-th power.
\end{lemma}

\begin{proof} Let $a(n)$ be the function  defined by (\ref{andfn}). By (\ref{lddfn}), we have $\zeta(ms)=\zeta(s)\sum_{n=1}^\infty \frac{\lambda_m(n)}{n^s}$, which implies that $a(n)=\sum_{d|n}\lambda_m(d)$.  Note that $\lambda_m(n)$ is a multiplicative function. 
	Following \cite{a77}, for  $n=p_1^{\alpha_1}\cdots p_r^{\alpha_r}, p_1<\cdots<p_r$, 
	we have
	\begin{align}
		\sum_{d|n}\lambda_m(d)f(p(d))&=\lambda_m(1)f(1)+\sum_{j=1}^rf(p_j)\sum_{{d|n, p(d)=p_j}}\lambda_m(d)\nonumber\\
		&=\sum_{j=1}^rf(p_j)\sum_{k=1}^{\alpha_j}\sum_{e|d_{j+1}}\lambda_m(p_j^ke)\nonumber\\
		&=\sum_{j=1}^rf(p_j)\big(\sum_{k=1}^{\alpha_j}\lambda_m(p_j^k)\big)\sum_{e|d_{j+1}}\lambda_m(e)\nonumber\\
		&=\sum_{j=1}^rf(p_j)\big(a(p_j^{\alpha_j})-1\big)a(d_{j+1})\label{dualpfeq}
	\end{align}
	where  $d_j=p_j^{\alpha_j}p_{j+1}^{\alpha_{j+1}}\cdots p_r^{\alpha_r}$ for $1\leq j\leq r$ and $d_{r+1}=1$.
	
	Let $j_0$ be the largest index $j$ such that $m\nmid \alpha_{j}$. Then $a(p_j^{\alpha_j})=1$ for $j>j_0$, $a(d_{j+1})=1$ for $j\geq j_0$ and $a(d_{j+1})=0$ for $j<j_0$. The sum (\ref{dualpfeq}) turns out to be $-f(p_{j_0})$ and  (\ref{dualeq}) follows.
\end{proof}

\begin{remark}  Similarly, one can prove that for  $n=p_1^{\alpha_1}\cdots p_r^{\alpha_r}, p_1<\cdots<p_r$, 
	$$\sum_{d|n}\lambda_m(d)f(P_m(d))=-\sum_{j=1}^{j_0}f(p_j)d_j(n)$$
	where $d_j(n)=\sum\limits_{\smat{m|\alpha_1,\dots,m|\alpha_{j-1},\\d^m|p_j^{\alpha_j-1}p_{j+1}^{\alpha_{j+1}}\cdots p_r^{\alpha_r}}}1$ and  ${j_0}$  is the first index $j$ such that $m\nmid \alpha_{j}$.
	
\end{remark}

\section{Proof of Theorem \ref{mainthm}}

\begin{theorem}[{\cite[Theorem (1.7)]{ip84}}]\label{orderthm} 
	Let $P(n)$ be the largest prime divisor of $n$.
	Then for $r>-1$,
	\begin{equation}
	\sum_{\smat{n\leq x\\P(n)^2|n}}\frac1{P(n)^r}=
	x\exp\Big\{-(2r+2)^{\frac12}(\log x\log^{(2)} x)^{\frac12}\Big(1+g_r(x)+O\Big(\Big(\frac{\log^{(3)} x}{\log^{(2)} x}\Big)^3\Big)\Big)\Big\}
	\end{equation}
	where $\log^{(k)}x=\log(\log^{(k-1)}x)$ is the $k$-fold iterated natural logarithm of $x$ and  
	$$g_r(x)=\frac{\log^{(3)}x+\log(1+r)-2-\log2}{2\log^{(2)}x}\big(1+\frac2{\log^{(2)}x}\big)-\frac{\big(\log^{(3)}x+\log(1+r)-2\big)^2}{8(\log^{(2)}x)^2}.$$
\end{theorem}

\begin{corollary}\label{maincor}
	There exists some constant $C_m$ such that
	\begin{equation}\label{maincoreq}
	\sum_{\smat{n\leq x\\P_m(n)\neq P(n)}}1=O(x\exp(-c(\log x\log^{(2)}x)^{\frac12}))
	\end{equation}
	and 
	\begin{equation}\label{maincoreq2}
	\sum_{\smat{n\leq x\\P_m(n)\neq P(n)}}\frac1n=C_m+O(\exp(-c(\log x\log^{(2)}x)^{\frac12})),
	\end{equation}
	where $c>0$ is a positive constant.
\end{corollary}

\begin{proof} 
	Equation (\ref{maincoreq}) follows by the case $r=0$ in Theorem \ref{orderthm}. 
	
	Put $e(x)=	\sum\limits_{\smat{n\leq x\\P_m(n)\neq P(n)}}1$. Then (\ref{maincoreq2}) can be deduced by (\ref{maincoreq}) as follows
	$$
	\sum_{\smat{n\leq x\\P_m(n)\neq P(n)}}\frac1n=\int_1^x\frac{de(t)}{t}=\left.\frac{e(t)}{t}\right|_1^x+\int_1^x\frac{e(t)dt}{t^2}=C_m-\int_x^\infty\frac{e(t)dt}{t^2}+\frac{e(x)}{x},
	$$
	where $C_m=\int_1^\infty\frac{e(t)dt}{t^2}$.
\end{proof}

\begin{remark}
	Due to this corollary, $P_m(n)$ inherits a lot of properties of $P(n)$. For example, one can get a version of Theorem \ref{orderthm} for $P_m(n)$. Another example we would like to mention is  that  $P_m(n)$ is equi-distributed $(\on{mod} k)$ for $k\geq2$ by Theorem 1 in \cite{a77}.
\end{remark}

Now we  prove Theorem \ref{mainthm} by showing the following theorem.

\begin{theorem} \label{mainthm2}
	Under the notation and assumptions of Theorem \ref{mainthm}, we have
	\begin{equation}\label{maineq2}
	-\sum_{\smat{2\leq n \leq x\\ \left[\frac{K/\Q}{p(n)}\right]=C}}\frac{\lambda_m(n)}{n}=\frac{|C|}{|G|}+O\big(\exp(-c(\log x)^{\frac13})\big),
	\end{equation}
	where $c$ is a positive constant.
\end{theorem}

\begin{proof} Here we follow the ideas in the proof 2 of  \cite[Theorem 4]{a77} and the proof of  \cite[Theorem 1]{d17}. 
	
	Let $f(n)$ be an arithmetic function defined by
	$$f(n)=\left\{ \begin{matrix}
	1,& \text{if } \left[\frac{K/\Q}{p}\right]=C, n=p>1;\\
	0,& \text{otherwise}.
	\end{matrix}\right.$$
	Then 
	$$\sum_{\smat{2\leq n \leq x\\ \left[\frac{K/\Q}{p(n)}\right]=C}}\frac{\lambda_m(n)}{n}=\sum_{n\leq x}\frac{\lambda_m(n)f(p(n))}{n}.$$
	
	As \cite[(2.35)]{a77}, by the M\"{o}bius inversion formula and the Duality Lemma \ref{duality} we have
	\begin{align}
		\quad\sum_{n\leq x}\frac{\lambda_m(n)f(p(n))}{n}&=-\sum_{n\leq x}\frac1n\sum_{d|n}\mu(\frac{n}d)f(P_m(d))=-\sum_{nd\leq x}\frac{\mu(n)}{n}\cdot\frac{f(P_m(d))}{d}\nonumber\\
		&=-\sum_{n\leq x^{\frac12}}\frac{\mu(n)}{n}\sum_{d\leq\frac{x}{n}}\frac{f(P_m(d))}{d}-\sum_{n< x^{\frac12}}\frac{f(P_m(n))}{n}\sum_{x^{\frac12}<d\leq\frac{x}{n}}\frac{\mu(d)}{d}
	\end{align}
	It follows that the difference between the partial sums on $\lambda_m$ and $\mu$ is
	\begin{align}\label{pfeq1}
		&	\sum_{\smat{2\leq n \leq x\\ \left[\frac{K/\Q}{p(n)}\right]=C}}\frac{\lambda_m(n)}{n}-\sum_{\smat{2\leq n \leq x\\ \left[\frac{K/\Q}{p(n)}\right]=C}}\frac{\mu(n)}{n}=-\sum_{n\leq x^{\frac12}}\frac{\mu(n)}{n}\sum_{d\leq\frac{x}{n}}\frac{f(P_m(d))-f(P(d))}{d}\nonumber\\
		&\qquad\qquad\qquad\qquad\quad-\sum_{n< x^{\frac12}}\frac{f(P_m(n))-f(P(n))}{n}\sum_{x^{\frac12}<d\leq\frac{x}{n}}\frac{\mu(d)}{d}=S_1+S_2
	\end{align}
	
	For $S_2$, by  \cite[(2.24)]{a77} we have
	\begin{equation}\label{pfeq0}
	\sum_{n\leq x^{\frac12}}\frac{\mu(n)}{n}=O\big(\exp(-c(\log x)^{\frac12})\big),
	\end{equation}
	and so we get that
	\begin{equation}
	\sum_{x^{\frac12}<d\leq\frac{x}{n}}\frac{\mu(d)}{d}=O\Big(\exp\big(-c(\log \frac{x}{n})^{\frac12}\big)\Big).
	\end{equation}
	As \cite[(2.27)]{a77}, this implies that
	\begin{equation}\label{pfeq2}
	S_2=O\Big(\sum_{n< x^{\frac12}}\frac{1}{n}\exp\big(-c(\log \frac{x}{n})^{\frac12}\big)\Big)=O\big(\exp(-c(\log x)^{\frac12})\big)
	\end{equation}
	
	For $S_1$, by  (\ref{maincoreq2}) in Corollary \ref{maincor}, 
	\begin{equation}
	\sum_{d\leq\frac{x}{n}}\frac{f(P_m(d))-f(P(n))}{d}=C_m+O\big(\exp(-c(\log \frac{x}{n})^{\frac12})\big).
	\end{equation}
	Similar to (\ref{pfeq2}) and by (\ref{pfeq0}) again, we get that
	\begin{equation}\label{pfeq3}
	S_1=-C_m\sum_{n\leq x^{\frac12}}\frac{\mu(n)}{n}+O\big(\exp(-c(\log x)^{\frac12})\big)=O\big(\exp(-c(\log x)^{\frac12})\big).
	\end{equation}
	
	Thus, (\ref{maineq2}) follows by combining (\ref{pfeq1}), (\ref{pfeq2}), (\ref{pfeq3}) and \cite[(10)]{d17} together.
\end{proof}

\begin{remark}	
	Similar to the proof of Theorem \ref{mainthm2}, one can also prove the analogues of formula  (\ref{classical}) and (\ref{maineq}) for  functions $(-1)^{\omega(n)}$ and $(-1)^{A(n)}$, where $\omega(n)=\sum_{p^\alpha||n}1$  is the prime divisor counting function and  $A(n)=\sum_{p^\alpha||n}\alpha p$ is the additive prime divisor function which was introduced by Alladi and Erd\"{o}s \cite{ae77} in 1977. This is mainly due to the Duality Lemma \ref{duality} with respect to $(-1)^{\omega(n)}$ and $(-1)^{A(n)}$ holds for the numbers $n$ satisfying $P(n)||n$ and  $P(n)\geq3$. 
\end{remark}

\section*{Acknowledgments}

The author would like to thank his advisor Professor Xiaoqing Li for recommending the article which leads him to write down this note. The author would also like to thank the anonymous referee for the very detailed comments, corrections and valuable suggestions, which improve this note a lot.

\end{document}